\documentclass[a4paper,11pt]{amsart}

\usepackage{amsthm,amsmath,amssymb}
\usepackage{mathtools}
\usepackage[utf8]{inputenc}
\usepackage[linktocpage]{hyperref}
\usepackage{multicol}
\hypersetup{colorlinks=true,linkcolor=blue,citecolor=blue,filecolor=magenta,urlcolor=blue}
\usepackage{nicematrix}
\usepackage{tikz}
\usepackage{pgf}
\usetikzlibrary{calc, cd, decorations.markings}
\tikzset{->-/.style={decoration={markings, mark=at position #1 with {\arrow{>}}},postaction={decorate}}}

\usepackage{enumitem}

\theoremstyle{plain}

\newtheorem{theorem}{Theorem}[section]
\newtheorem{prop}[theorem]{Proposition}
\newtheorem{cor}[theorem]{Corollary}
\newtheorem{pty}[theorem]{Property}
\newtheorem{lemma}[theorem]{Lemma}

\theoremstyle{definition}

\newtheorem{ex}[theorem]{Example}

\theoremstyle{remark}
\newtheorem{remark}[theorem]{Remark}

\def\bz{\mathbb{Z}}

\def\bq{\mathbb{Q}}
\def\bc{\mathbb{C}}
\def\bp{\mathbb{P}}

\DeclareMathOperator{\GL}{GL}
\DeclareMathOperator{\pgl}{PGL}

\DeclareMathOperator{\aut}{Aut}

\numberwithin{equation}{section}

\title{Algebraic and symplectic curves of degree 8}

\author[E.~Artal]{Enrique Artal Bartolo}
\address[E.~Artal Bartolo]{Departamento de Matem\'{a}ticas, IUMA \\
Universidad de Zaragoza \\
C.~Pedro Cerbuna 12, 50009, Zaragoza, Spain}
\urladdr{\url{http://riemann.unizar.es/~artal}}
\email{\href{mailto:artal@unizar.es}{artal@unizar.es}}

\thanks{Partially supported by 
Grant PID2020-114750GB-C31 funded by MCIN/AEI/10.13039/\ 501100011033
and 
Departamento de Ciencia, Universidad y Sociedad del Conocimiento of the Gobierno de Arag{\'o}n
(E22\_20R: ``{\'A}lgebra y Geometr{\'i}a'')}

\begin{document}

\begin{abstract}
We study the existence of some irreducible projective plane curves of degree~$8$
with some prescribed topological type of singularities in the algebraic
and symplectic worlds.
\end{abstract}

\maketitle

\section*{Introduction}
Since the eighties the study of the theory of complex analytic and algebraic varieties have been enriched
by the study of pseudo-holomorphic and complex symplectic varieties. In the case of curves in the projective
plane these new objects are strongly related to braid monodromy, see \cite{Orevkov2000, Kulikov:07, KhKu:03}, and can be constructed 
by local deformations of arrangements of algebraic plane curves which can be expressed in terms of braid monodromy factorizations
which are locally algebraic.

The starting point of this paper is an unpublished idea of S.Yu.~Orevkov, which is explained in~\cite{Golla-Starkston:22} and outlined
in \S\ref{sec:symp}. The main idea is to deform symplectically a tricuspidal quartic (or deltoid) such that the tangent lines
to the cusps are not concurrent. The idea of Orevkov, performed in detail by M.~Golla and L.~Starkston is to apply a standard
Cremona transformation in order to obtain an irreducible symplectic curve with a configuration of topological type of singularities
which does not exist in the algebraic category. 

In this work, we replace the Cremona transformation by a Kummer cover,
in order to compare the symplectic and algebraic structures of curves of degree~$4n$ with $3n$ singularities having
the topological type of $u^{2n} - v^3=0$. The case $n=2$ offers significant interesting properties and we focus our attention
in this case since the study of algebraic structures seems to be cumbersome for $n>2$. We prove the existence
of symplectic curves $C_\text{\rm symp}$ of degree~$8$ with~$6$ singular points of type $\mathbb{E}_6$.

For $n=2$ the primary goal is to determine all algebraic curves of degree~$8$ with~$6$ singular points 
with the topological type of $\mathbb{E}_6$. 
Unfortunately, the goal was too ambitious and has not been reached. As it usually happens, the existence of symmetries is helpful
and in this paper we determine all such curves fixed by a non-trivial projective automorphism. There is exactly one such curve 
$C_{8,2}$
(up to projective automorphism, of course) fixed by an involution and four such curves $C_{8,3}^i$, $i=1,\dots,r$, invariant by an automorphism of order~$3$;
there are no more curves invariant by automorphism. These four curves have equations in conjugate number fields $\mathbb{K}_i\subset\bc$ isomorphic
to $\bq[t]/p(t)$ where $p(t)$ is an irreducible polynomial of degree~$4$. A main question is if they share topological properties. Two of the roots of $p(t)$ are real and two complex conjugate; in this last case, complex conjugation is a homeomorphism of $\bp^2$ reversing orientations on the curves. 
In the general case, most likely these curves are rigid by dimension arguments. 

Another result in this paper is that there is no homeomorphism of $\bp^2$ sending $C_\text{\rm symp}$ to an algebraic symmetric curve, but it may be isotopic to a non symmetric one (if such a curve exists). There is also no homeomorphism of $\bp^2$ sending $C_{8,2}$ to one of the 
$C_{8,3}^i$, and, besides complex conjugation, we do not know the existence of homeomorphism of $\bp^2$ exchanging the curves $C_{8,3}^i$ (respecting or reversing orientations).

Some proofs need non-straightforward computer algebra steps and rely heavily on computations in \texttt{Sagemath}~\cite{Sagemath}.
The steps are described in several notebooks located in \url{https://github.com/enriqueartal/SymplecticOctics} which can be executed either in a computer with the last version of \texttt{Sagemath} or online using \texttt{Binder}~\cite{binder}.

In \S\ref{sec:deltoid} we describe some known properties of the deltoid and compute a special presentation
of the fundamental group of the complement of the deltoid and the tangent lines at the cusps.
In \S\ref{sec:symp} we study the topology of a symplectic deformation of the previous arrangement
of curves.

This paper is inspirated in the ideas fruitfully discussed in the workshop \emph{Complex and symplectic curve configurations} organized by M.~Golla, P.~Pokorav and L.~Starkston in Nantes, December 2022, and in particular
in~\cite{Golla-Starkston:22}. I would like to thank the organizers, the attendants, and specially the team work on braid monodromy and symplectic curves.

\section{The deltoid and its tangents at the singular points}
\label{sec:deltoid}

The deltoid (or tricuspidal quartic, i.e, plane quartic with three ordinary cusps) 
is an important plane projective curve. It is rigid, in the sense,
that two deltoids are projectively isomorphic. As it is the dual 
of a nodal cubic, it has the following well-known property.

\begin{pty}
The three tangent lines at the cusps of a deltoid are concurrent lines.
\end{pty}

A symmetric equation of the deltoid is 
\[
y^2 z^2 + z^2 x^2 + x^2 y^2 - 2 x y z(x + y +z)=0.
\]
The equation of the curve in the right-hand side of Figure~\ref{fig:deltoide} 
is
\[
v^{4} + 4 (1 + u) v^{3} + 18 u v^{2} - 27 u^{2}=0;
\]
the line at infinity is the tangent line to one of the cusps; the other
cusps are $(0,0), (1, -3)$ and they have vertical tangent lines. 
The vertical lines $u=a$, $a\in\mathbb{R}$, intersect the real part of the curve at 
two real points (solid curves in the right-hand side of Figure~\ref{fig:deltoide}) and at two other points $(a, v_0(a)\pm \sqrt{-1}v_1(a))$, $v_0(a)\in\mathbb{R}$, $v_1(a)\in\mathbb{R}_{>0}$; the dotted
curve in the right-hand side of Figure~\ref{fig:deltoide} represents $(a,v_0(a))$.
This picture provides a topological model of the curve
and its tangents.

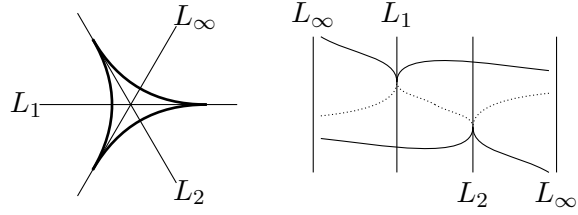
\begin{figure}[ht]
\centering
\begin{tikzpicture}
\newcommand\rama[3]{\draw[line width=1, rotate=#1] (0: 1) to[out=180, in=-60] (120: 1);
\draw[rotate=#1] (180: 1.2) node[#3=-5pt] {#2} -- (0: 1.4);}
\newcommand\deltoidereal[3]{\draw[scale=#1] (-3, 3) to[out=-45, in=90] (-1, 1)
to[out=90, in=-190] (3, 1.5);
\draw[scale=#1] (-1, -3) -- (-1, 3) node[#3] {#2};
\draw[scale=#1] (-3.2, -3) -- (-3.2, 3) node[#3] {$L_\infty$};
}
\newcommand\deltoideimag[1]{\draw[densely dotted, scale=#1] (-3, -.5) to[out=10, in=-90] (-1, 1)
to[out=-90, in=135] (0,0);}
\begin{scope}
\rama{0}{$L_1$}{left}
\rama{120}{$L_2$}{below right}
\rama{-120}{$L_\infty$}{above right}
\end{scope}
\begin{scope}[xshift=4cm, xscale=.5, yscale=.3]
\deltoidereal{1}{$L_1$}{above}
\deltoidereal{-1}{$L_2$}{below}
\deltoideimag{1}
\deltoideimag{-1}
\end{scope}
\end{tikzpicture}
 \caption{Left: usual deltoid. Right: real picture with real parts of non real branches.}
\label{fig:deltoide}
\end{figure}

Using the techniques of \cite{ACCT}, applied to the right-hand side
of Figure~\ref{fig:deltoide}, we obtain the following result.
The justification of this figure can be found in the notebook 
\texttt{ConstructionSymplecticGroup}.

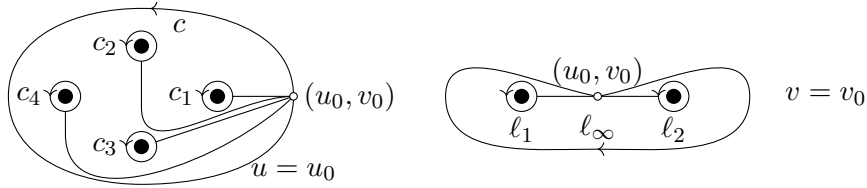
\begin{figure}[ht]
\begin{tikzpicture}

\begin{scope}
\coordinate (B) at (0,0);
\foreach \x in {1,4}
{
\coordinate (A\x) at ({2*(1-\x)/3-1},0);
}
\foreach \x in {2,3}
{
\coordinate (A\x) at (-2,{2*(5-2*\x)/3});
}

\draw (B) -- (A1);
\draw (B) to[out=-180, in=-90]
(-2,-1/7) to (A2);
\draw (B) -- (A3);
\draw (B) to[out=-135, in=-90]
(-3,-1/2) to (A4);

\foreach \x in {1,...,4}
{
\draw[fill=white,->-=.5] (A\x) circle [radius=.2]
node[left=5pt] {$c_\x$};
\fill (A\x) circle [radius=.1];
}

\draw[->-=.25] (B) to[out=90, in=0] node[below right, pos=.9] {$c$} ($(A2)+(0,.5)$)
to[out=180, in=90] ($(A4)+(-.75,0)$)
to[out=-90, in=180] ($(A3)+(0,-.5)$)
to[out=0, in=-90] (B);

\draw[fill=white] (B) circle [radius=.05] node[right] {$(u_0,v_0)$};
\node at (0,-1) {$u=u_0$};

\end{scope}

\begin{scope}[xshift=4cm]
\coordinate (B) at (0,0);

\coordinate (A1) at (-1,0);
\coordinate (A2) at (1,0);

\draw (B) -- (A1);
\draw (B) -- (A2);

\foreach \x in {1,2}
{
\draw[fill=white,->-=.5] (A\x) circle [radius=.2]
node[below=4pt] {$\ell_\x$};
\fill (A\x) circle [radius=.1];
}

\draw[->-=.5] (B) to[out=10,in=90] ($2*(A2)$)
to[out=-90, in=0] (0, -.7) node[above] {$\ell_\infty$} to[out=180,in=-90] ($2*(A1)$)
to[out=90, in=170] (B);

\draw[fill=white] (B) circle [radius=.05] node[above] {$(u_0,v_0)$};
\node at (3,0) {$v=v_0$};
\end{scope}
\end{tikzpicture}

\caption{Generators of the presentation in Proposition~\ref{fg-deltoide-alg}.
The base point is $(u_0,v_0)$, where $0<u_0<1$ (the $u$-coordinates) of the affine
singular points, and $v_0\gg 0$.}
\end{figure}

\begin{prop}\label{fg-deltoide-alg}
The braid monodromy of the deltoid projecting from the intersection point of the 
tangent line to the cusps, when one of these tangents is the line at infinity
(as in Figure{\rm~\ref{fig:deltoide}}, right) is given by
$(\sigma_2\cdot\sigma_1)^2$ (for $L_1$) and $(\sigma_2\cdot\sigma_3)^2$ (for $L_2$).

As a consequence, the fundamental group $G_{\Delta L}$ of the complement of the deltoid
and the lines $L_1,L_2,L_\infty$ is generated by $c_1,\dots,c_4,\ell_1,\ell_2,\ell_\infty$
with the relations
\begin{enumerate}[label=\rm(R\arabic{enumi}), series=relations1]
\item\label{rel:1} $[\ell_2, c_1]=1$
\item\label{rel:2} $\ell_2^{-1}\cdot c_2\cdot\ell_2=(c_2\cdot c_3\cdot c_4)\cdot c_3\cdot (c_2\cdot c_3\cdot c_4)^{-1}$
\item $\ell_2^{-1}\cdot c_3\cdot\ell_2=(c_2\cdot c_3)\cdot c_4\cdot (c_2\cdot c_3)^{-1}$
\item\label{rel:4} $\ell_2^{-1}\cdot c_4\cdot\ell_2=c_2$
\item\label{rel:5} $\ell_1^{-1}\cdot c_1\cdot\ell_1=(c_1\cdot c_2)\cdot c_3\cdot (c_1\cdot c_2)^{-1}$
\item $\ell_1^{-1}\cdot c_2\cdot\ell_1=(c_1\cdot c_2)\cdot c_1\cdot (c_1\cdot c_2)^{-1}$
\item\label{rel:7} $\ell_1^{-1}\cdot c_3\cdot\ell_1=c_1\cdot c_2\cdot c_1^{-1}$
\item\label{rel:8} $[\ell_1, c_4]=1$
\item\label{rel:infinity} $c\cdot\ell_1\cdot\ell_2\cdot\ell_\infty=1$
\end{enumerate}
where $c=c_1\cdot\ldots\cdot c_4$.
\end{prop}

This monodromy can be also computed using \texttt{Sagemath}~\cite{Sagemath} with the optional package 
\texttt{Sirocco}~\cite{Marco2016}, but in this case it can be done directly. 

In the presentation of $G_{\Delta L}$ we may omit the generator $\ell_\infty$ using
\ref{rel:infinity} which comes from the situation
at infinity. Actually, Zariski-van Kampen method can be thought to happen in the
blow-up of the projection point (the point at infinity of the vertical lines), see Figure~\ref{fig:infinito}. Then \ref{rel:infinity} comes 
from the boundary of a neighbourhood of the exceptional divisor $E$, see \cite{Mumford1961, ACMat:2020}. Note that
the \emph{natural} meridian $e$ of $E$ is the inverse of $c$. The normal crossing situation
implies that~$e$ (and hence $c$ and $\ell_1\cdot\ell_2\cdot\ell_\infty$) commute with $\ell_1,\ell_2,\ell_\infty$.
This is a consequence of \ref{rel:2}-\ref{rel:7}: $\ell_1,\ell_2$ commute with $c$ as
their conjugation action comes from braids.

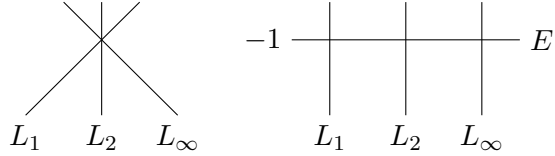
\begin{figure}[ht]
\centering
\begin{tikzpicture}
\newcommand\deltoidereal[3]{\draw[scale=#1] (-3, 3) to[out=-45, in=90] (-1, 1)
to[out=90, in=-190] (3, 1.5);
\draw[scale=#1] (-1, -3) -- (-1, 3) node[#3] {#2};
\draw[scale=#1] (-3.2, -3) -- (-3.2, 3) node[#3] {$L_\infty$};
}
\newcommand\deltoideimag[1]{\draw[densely dotted, scale=#1] (-3, -.5) to[out=10, in=-90] (-1, 1)
to[out=-90, in=135] (0,0);}
\begin{scope}
\draw (-1/2, 1/2) -- (1, -1)node[below] {$L_\infty$} ;
\draw (0, 1/2) -- (0, -1)  node[below] {$L_2$} ;
\draw (1/2, 1/2) -- (-1, -1) node[below] {$L_1$} ;
\end{scope}
\begin{scope}[xshift=4cm,]
\draw (1, 1/2) -- (1, -1)node[below] {$L_\infty$} ;
\draw (0, 1/2) -- (0, -1)  node[below] {$L_2$} ;
\draw (-1, 1/2) -- (-1, -1) node[below] {$L_1$} ;
\draw (-3/2, 0) node[left] {$-1$} -- (3/2, 0) node[right] {$E$};
\end{scope}
\end{tikzpicture}
 \caption{Blow-up of the projection point.}
\label{fig:infinito}
\end{figure}

\begin{remark}\label{sdfree}
Note that $G_{\Delta L}$ is a semidirect product $\mathbb{F}_4\rtimes\mathbb{F}_2$
where $c_1,\dots,c_4$ are the generators of normal subgroup $\mathbb{F}_4$, $\ell_1,\ell_2$
are the generators of $\mathbb{F}_2$ and \ref{rel:2}-\ref{rel:7}, determine the conjugation action.
\end{remark}

This group has been computed in \cite{ADT:10}, but we need the above computation both for
completeness and to deal with the symplectic deformations.

\section{Symplectic deformations}
\label{sec:symp}

In the context of symplectic geometry, it is possible to
construct a \emph{deltoid} for which the pseudo-holomorphic
tangent lines at the cusps are not concurrent. 
This was communicated long time ago to the author by S.Yu.~Orevkov and was formally
written in~\cite[\S~8]{Golla-Starkston:22}. Moreover
it can be done as a deformation of the algebraic curve
which is an isotopy outside a neighbourhood of the triple point.

\begin{figure}[ht]
\centering
\begin{tikzpicture}
\newcommand\deltoidereal[3]{\draw[scale=#1] (-3, 3) to[out=-45, in=90] (-1, 1)
to[out=90, in=-190] (3, 1.5);
\draw[scale=#1] (-1, -3) -- (-1, 3) node[#3] {#2};
\draw[scale=#1] (-3.2, -3) -- (-3.2, 3) node[#3] {$L_\infty$};
}
\newcommand\deltoideimag[1]{\draw[densely dotted, scale=#1] (-3, -.5) to[out=10, in=-90] (-1, 1)
to[out=-90, in=135] (0,0);}
\begin{scope}
\draw (-1/2, 1/2) -- (1, -1)node[below] {$L_\infty$} ;
\draw (0, 1/2) -- (0, -1)  node[below] {$L_2$} ;
\draw (1/2, 1/2) -- (-1, -1) node[below] {$L_1$} ;
\end{scope}
\begin{scope}[xshift=4cm,]
\draw (-1/2, 1/2) -- (1, -1)node[below] {$L_\infty$} ;
\draw (1/5, 1/2) -- (0, -1)  node[below] {$L_2$} ;
\draw (1/2, 1/2) -- (-1, -1) node[below] {$L_1$} ;
\end{scope}
\end{tikzpicture}
 \caption{Symplectic deformation of an ordinary triple point.}
\label{fig:deformation}
\end{figure}
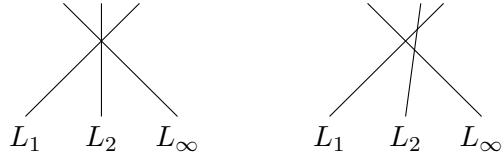

Using the classical Seifert-van Kampen theorem, the fundamental
group of the complement of this symplectic curve has
the same presentation of the algebraic one, adding 
the relations from the situation in the right-hand side of
Figure~\ref{fig:deformation}, i.e.,
$[\ell_1,\ell_2]=[\ell_1,\ell_\infty]=[\ell_2,\ell_\infty]=1$.
This technique has been used in \cite{ACM:2020b, AC:1998, CatWaj:05}
Actually, the following holds.

\begin{cor}\label{fg-deltoide-symp}
The fundamental group $G_{s\Delta L}$ of the complement of the \emph{symplectic} deltoid
and the tangent lines at the cusps has the generators and relators
of Proposition{\rm~\ref{fg-deltoide-alg}} plus the relation
\begin{enumerate}[label=\rm(R\arabic{enumi}), resume=relations1]
\item\label{rel:conm} $[\ell_1,\ell_2]=1$.
\end{enumerate}
\end{cor}

It is useful to have a semidirect presentation of this group.

\begin{cor}\label{cor:sdelta}
The group $G_{s\Delta L}$ is a semidirect product
$G_0\rtimes\mathbb{Z}^2$ where the action is as in the algebraic case
and 
\[
G_0 = \langle c_1,\dots,c_4\mid 
c_3\cdot c_4\cdot c_3 = c_4\cdot c_3\cdot c_4,\ 
c_1\cdot c_2\cdot c_1 = c_2\cdot c_1\cdot c_2\ 
\rangle.
\]
\end{cor}

\begin{proof}
We start with the semidirect product structure $G_{\Delta L}=\mathbb{F}_4\rtimes\mathbb{F}_2$
and the natural epimorphism $G_{\Delta L}\twoheadrightarrow G_{s\Delta L}$:
\[
\begin{tikzcd}
1\rar&\mathbb{F}_4\rar\dar&G_{\Delta L}\rar\dar&\mathbb{F}_2\dar\rar\lar[bend right=30]&1\\
1\rar&G_0\rar&G_{s\Delta L}\rar&\mathbb{Z}^2\rar\lar[bend left=30]&1.
\end{tikzcd}
\]
The semidirect structure comes from the fact that the below exact sequence splits as seen in the 
relations. The conjugation in the first exact sequence is given by the action of the braids
$\tau_1:=(\sigma_2\cdot\sigma_1)^2$ (for $\ell_1$) and $\tau_2:=(\sigma_2\cdot\sigma_3)^2$ (for $\ell_2$).
Since $\ell_1,\ell_2$ commute in $G_{s\Delta L}$, then in $G_0$ we have the relations
(checked in \texttt{ConstructionSymplecticGroup})
\[
c_i^{\tau_1\cdot\tau_2} = c_i^{\tau_2\cdot\tau_1},\qquad i=1,\dots,4,
\]
which translates into the relations in the statement.
\end{proof}

\section{Algebraic and symplectic Cremona transformations}\label{sec:Orevkov}

Following the ideas of Orevkov, Golla and Starkston formalized in \cite[\S~8]{Golla-Starkston:22} an example of 
rational singular curves which exist in the symplectic category and not in the algebraic one.

The most well known birational is the map 
\[
\begin{tikzcd}[row sep=0,/tikz/column 1/.append style={anchor=base east},/tikz/column 2/.append style={anchor=base west}]
\mathbb{P}^2\rar[dashrightarrow]&\mathbb{P}^2\\
{[x:y:z]}\rar[mapsto]&{[y z : z x : x y]}.
\end{tikzcd}
\]
Geometrically is obtained by the blow-up of the points $[1:0:0],[0:1:0],[0:0:1]$ and the 
blow-down of the strict transforms of the lines $x=0,y=0,z=0$ which are pairwise disjoint
smooth rational $(-1)$-curves. We can also consider a \emph{symplectic} Cremona transformation which
gives the following result.

\begin{prop}[{\cite[\S~8]{Golla-Starkston:22}}]
Let $\Sigma_{\text{\rm alg}}$ (resp. $\Sigma_{\text{\rm symp}}$) be the space of algebraic
(resp. symplectic) irreducible curves of degree~$8$ in $\mathbb{P}^2$ having three singular
points with the topological type of $u(v^3 + u^5)=0$.
\begin{enumerate}[label=\rm(\arabic{enumi})]
\item The space $\Sigma_{\text{\rm alg}}$ is empty.
\item The space $\Sigma_{\text{\rm symp}}$ is non-empty and it can be
embedded in the space of symplectic deltoids such that their tangent lines to the
cuspidal points are not concurrent.
\end{enumerate}
\end{prop}

We can go further and compute some topological invariants of this curve, in particular 
the fundamental group of its complement.

\begin{cor}
If $C\in\Sigma_{\text{\rm symp}}$ comes from a Cremona transformation associated 
to the tangent lines of a symplectic deltoid (isotopic to an algebraic one), then 
its fundamental group is the non-abelian semidirect product 
$\mathbb{Z}/3\rtimes\mathbb{Z}/8$.
\end{cor}

\begin{proof}
If $P$ is an ordinary double point and two commuting meridians of the branches, 
a meridian of the exceptional component of the blow-up of the $P$ is
the product of the meridians, see e.g. \cite[Lemma~3.6]{ACM-Nemethi} (probably
well-known result). 

The complement of $C$ is homeomorphic to the complement of the strict transform
of the deltoid and the tangent lines by the blow-ups. For the total transform,
we have to add the exceptional components. From the deformation
in Figure~\ref{fig:deformation}, we see that these meridians are 
$\ell_1\cdot\ell_2$, $\ell_1\cdot\ell_\infty$, and $\ell_2\cdot\ell_\infty$.

From \cite[Lemma 4.18]{Fujita1982}, the fundamental group of the complement
of the strict transform is obtained by \emph{killing} these meridians. 
These new relations are summarized in
\[
\ell:=\ell_1=\ell_2=\ell_\infty,\quad \ell^2=1
\]
and clearly imply \ref{rel:conm}. The relation~\ref{rel:infinity} becomes $c=\ell$.
Relations \ref{rel:1} and \ref{rel:8} become $[\ell,c_1]=[\ell,c_2]$; since 
\ref{rel:4} becomes $c_2=c_4$ we obtain also that $\ell$ is central.
From \ref{rel:5} we can eliminate $c_1$ and 
from a simple computation we obtain that $c_2,c_3$ generate with the relations
\[
c_2\cdot c_3\cdot c_2 = c_3\cdot c_2\cdot c_3,\quad c_2^2\cdot c_3^2
\text{ central and of order }2.
\]
The normal subgroup of order~$3$ is generated by $c_2\cdot c_3^{-1}=(c_2\cdot c_3)^4$ and 
the subgroup of order~$8$ is generated by $c_2$.
\end{proof}

\begin{remark}
This group is also the fundamental group of the complement of an algebraic curve, as 
it is shown using similar techniques in~\cite{Uludag:2001}.
\end{remark}

\section{Kummer covers}
\label{sec:kummer}

With the same ideas of \S\ref{sec:Orevkov}, we are going to construct new examples
replacing the standard Cremona transformation by Kummer covers, i.e., Galois covers
\[
\begin{tikzcd}[row sep=0,/tikz/column 1/.append style={anchor=base east},/tikz/column 2/.append style={anchor=base west}]
\mathbb{P}^2\rar[]&\mathbb{P}^2\\
{[x:y:z]}\rar[mapsto]&{[x^n : y^n : z^n]}.
\end{tikzcd}
\]
Starting from three pseudo-holomorphic non-concurrent lines there is a symplectic counterpart. 

Let us recall that an $\mathbb{E}_6$-singularity is a germ of plane curve singularity in a smooth surface
isomorphic to $u^3-v^4=0$ in $(\mathbb{C}^2,0)$ (with local coordinates $u,v$).

\begin{prop}\label{prop:group_symp}
There are irreducible symplectic curves $C_{\text{\rm symp}}$ of degree~$8$ in $\bp^2$ with $6$ singular points
of type $\mathbb{E}_6$ for which the fundamental group $G_{\text{\rm symp}}$ of their complement 
is generated by $c'_1,c_2,c_3,c_4$,
$c'_1=c_2^{-1}\cdot c_1\cdot c_2$, with relations
\begin{gather*}
[c_2, c_4]=[c'_1,c_3]=1,\ 
c'_1\cdot c_2\cdot c'_1 = c_2\cdot c'_1\cdot c_2,\ 
c_3\cdot c_2\cdot c_3 = c_2\cdot c_3\cdot c_2, \\
c_3\cdot c_4\cdot c_3 = c_4\cdot c_3\cdot c_4,\ 
(c_2\cdot c'_1\cdot c_3\cdot c_4)^2=1,
\end{gather*}
and the conjugation action is derived from the action in Remark{\rm~\ref{sdfree}}.
\end{prop}

As in \S\ref{sec:Orevkov}, we denote by $\Lambda_{\text{\rm symp}}$ and 
$\Lambda_{\text{\rm alg}}$ the spaces of symplectic or algebraic curves of
degree~$8$ having $6$ singular points of type $\mathbb{E}_6$.

\begin{proof}
The existence of such a curve $C_{\text{\rm symp}}$ comes from a symplectic Kummer cover for~$n=2$, starting from 
a symplectic deltoid as in \S\ref{sec:symp}, taking the tangent lines
to the cusps for the ramification lines of the Kummer cover. The degree
of the preimage of the deltoid is~$8$ and each cusp produces 
two $\mathbb{E}_6$ points.

For the fundamental group, let $G_{\textrm{orb22}}$ be the orbifold fundamental group
of the complement of the deltoid where the orbifold structure comes 
from the action of $\bz/2\times\bz/2$ as deck group of the Kummer cover.

Hence, $G$ is the quotient of the group in Corollary~\ref{fg-deltoide-symp}
with some extra relations 
\begin{enumerate}[label=\rm(R\arabic{enumi}), resume=relations1]
\item $\ell_1^2=1$
\item $\ell_2^2=1$
\item\label{rel:orb_infinity} $\ell_\infty^2=1$
\end{enumerate}
As in the proof of Corollary~\ref{cor:sdelta},
this group $G$ is a semidirect product $G_{\text{\rm symp}}\rtimes\mathbb{Z}/2\times\mathbb{Z}/2$.
In order to find $G_{\text{\rm symp}}$ 
we consider the relations $c_i^{\tau_j}=c_i^{\tau_j^{-1}}$, for $i=1,\dots,4$ and $j=1,2$.
Moreover we can combine the relations \ref{rel:infinity} and \ref{rel:orb_infinity}
to rewrite them in terms only of $c_1,\dots,c_4$. Replacing $c_1$ by $c_1'$ we obtain 
the relation of the statement. Details can be found in \texttt{ConstructionSymplecticGroup}.

Since the fundamental 
group of the complement of $C_{\text{\rm symp}}$ is the kernel of the epimorphism
$G\to\bz/2\times\bz/2$ given by
\[
c_i\mapsto (0,0),\quad \ell_1\mapsto(1,0),\quad \ell_2\mapsto(0,1),\quad \ell_\infty\mapsto(1,1),
\]
we obtain that this group is $G_{\text{\rm symp}}$.
\end{proof}

\begin{remark}
Using \texttt{GAP4}~\cite{GAP} via \texttt{Sagemath}~\cite{Sagemath} we have:
\[
G/G'\cong\bz/8,\quad G'/G''\cong\bz/3,\quad 
G''/G'''\cong(\bz/2)^6,\quad G'''/G^{(4)}\cong\bz^9\oplus(\bz/2)^5\oplus\bz/4.
\]
\end{remark}

We need to understand $\Lambda_{\text{\rm alg}}$ in order to check if the elements found in $\Lambda_{\text{\rm symp}}$ are isotopic to 
algebraic curves. 
Unfortunately computations are cumbersome and 
our attempts failed. Most probably this space is discrete, and we have been able to obtain some particular
elements. Some geometric properties of these curves are presented in the following section.

\section{Symmetries and other properties of curves in \texorpdfstring{$\Lambda_{\text{\rm alg}}$}{Lambda alg}}
\label{sec:properties}

We want to study the properties of the curves in $\Lambda_{\text{\rm alg}}$. We start
with the symmetry properties. Let us recall that the automorphism group of
$\bp^2$ is the group $\pgl(3;\bc)$. The elements of finite order 
correspond to diagonalizable matrices (up to scalar multiplication)
whose eigenvalues are roots of unity.

\begin{ex}
The involutions of $\bp^2$ correspond to matrices which are conjugate
to the diagonal matrix $(1, 1, -1)$, i.e. conjugate 
to the automorphism $\Phi:\bp^2\to\bp^2$ such that $\Phi([x:y;z])=[x:y:-z]$.
These automorphisms have an isolated fixed point, $[0:0:1]$ (for the eigenspace
of dimension~$1$), and a line of fixed points,  $z=0$ (for the eigenspace
of dimension~$2$).
\end{ex}

The quotient of $\bp^2$ by an involution is isomorphic
to the weighted projective plane $\bp^2_{(1,1,2)}$.
Let $\omega:=(p, q, r)$ be a positive integer vector
with pairwise coprime weights and consider the 
weighted projective plane $\bp^2_\omega$, see \cite{Dolgachev1982} for details.
It is a normal projective surface structure in the quotient
\[
\bc^3\setminus\{\mathbf{0}\} / (x, y, z)\sim (t^p x, t^q y, t^r z);
\]
its elements are denoted $[x:y:z]_\omega$.
The curves are the zero loci of $\omega$-weighted homogeneous polynomials
and since they are in general Weil divisors. Bézout's formula 
is also valid in the weighted projective planes. Namely if 
$C_{1}, C_2$ are curves defined by 
$\omega$-weighted homogeneous polynomials of degrees $d_1,d_2$
respectively. Then
\begin{equation}\label{eq:Bezout}
C_1\cdot C_2 = \frac{d_1\cdot c_2}{p\cdot q\cdot r}.
\end{equation}

\begin{lemma}\label{lema-2}
Let $C\in\Lambda_{\text{\rm alg}}$ symmetric by the action of a projective involution~$\Phi_2$.
Then two of the singular points are in the line of fixed points in $\Phi_2$ and the other 
ones form two orbits.
\end{lemma}

\begin{proof}
We can assume that $\Phi_2([x:y:z]) = [x:y:-z]$; let $F_8(x,y,z)=0$ be the equation of $C$.
Since the curve is invariant, we have that $F_8(x,y,z)=F_8(x,y,-z)$, i.e., $F_8(x,y,z)=G_8(x,y,z^2)$
where $G_8$ is a $(1,1,2)$-weighted homogeneous polynomial of degree~$8$.
The quotient $\tilde{C}$ of $C$ is a curve in $\bp^2_{(1,1,2)}$
with equation $G_8(x_2, y_2, z_2)=0$.

An $\mathbb{E}_6$ point cannot be the isolated fixed point $[0:0:1]$ of $\Phi_2$.
Let us assume that no singular point is in the line of fixed points. Then, the quotient of $C$ in $\mathbb{P}^2_{(1,1,2)}$
is a curve of degree~$8$ with three triple points of type~$\mathbb{E}_6$.
There is no \emph{line} $L_1$ of equation $ax_2+by_2=0$ through two singular points. If it would be the case,
since $L_1$ is of degree~$1$, we would have
\[
4 = \frac{\deg L_1\cdot\deg\tilde{C}}{2}=L_1\cdot\tilde{C}\geq 3 + 3,
\]
so it is not possible. It is not difficult to check that three points in $\mathbb{P}^2_{(1,1,2)}$
such that no pair is contained in a line, are contained in a curve $C_2$ of degree~$2$. Then:
\[
8 = \frac{\deg C_2\cdot\deg\tilde{C}}{2}=C_2\cdot\tilde{C}\geq 3 + 3 + 3,
\] 
which is also impossible.
The only possible case is the one in the statement.
\end{proof}

\begin{ex}
There are two types of automorphisms $\bp^2$ of order~$3$.
The first one corresponds to matrices which are conjugate
to the diagonal matrix $(1, 1, -\zeta)$, where $\zeta:=\exp\frac{2\sqrt{-1}\pi}{3}$,
with one isolated fixed point and a line of fixed points.

The second type corresponds to matrices which are conjugate
to the diagonal matrix $(\zeta, \overline{\zeta}, 1)$ and 
has three isolated fixed points. There are exactly
three fixed lines, the lines joining the fixed points.
\end{ex}

\begin{lemma}\label{lema-3}
Let $C\in\Lambda_{\text{\rm alg}}$ symmetric by the action of a projective automorphism~$\Phi_3$
of order~$3$. Then $\Phi_3$ has no line of fixed points, there are~$2$ orbits and the curve 
passes through two isolated fixed points of $\Phi_3$ (tangent to the fixed lines
not containing the two fixed points in the curve).
\end{lemma}

\begin{proof}
Let us suppose first that $\Phi_3$ has a line of fixed points. At most two singular points
can be in this line by Bézout Theorem, but actually none of them can be in the line
since the orbits have one or three elements. But the points in the orbits are aligned
which is contradiction again with
Bézout Theorem.

Hence $\Phi_3$ has three fixed points, say $P_1, P_2,P_3$. These points cannot 
be singular points of the curve since an $\mathbb{E}_6$ cannot be fixed 
an isolated fixed point of an action of order~$3$.

Hence, the singular points form two orbits. Let us
consider the lines joining the fixed points, say $L_i$ is the line joining
$P_j$ and $P_k$, $\{i,j,k\}=\{1,2,3\}$. Since the action is free on $L_i\setminus\{P_j,P_k\}$,
it must intersect $C$ with intersection number~$6$. This is only achieved (after reordering)
if $P_1,P_2\in C$, $P_3\notin C$, $L_2$ is tangent to $C$ at $P_1$ and 
$L_1$ is tangent to $C$ at $P_2$.
\end{proof}

\begin{ex}
There are several types of automorphisms of order~$n>3$, depending on the different
configurations of eigenvalues.
\end{ex}

\begin{lemma}\label{lema-4}
There is no $C\in\Lambda_{\text{\rm alg}}$ symmetric by the action of a projective automorphism~$\Phi$
of order~$n>3$.
\end{lemma}

\begin{proof}
Note first that we cannot have a line of fixed points, only isolated points. 
The case $n>7$ is ruled out immediately. 

For $n=4$, there are two possible types of automorphisms~$\Phi_4$, conjugate to 
the diagonal matrices of either $(\sqrt{-1}, 1, 1)$ or $(\sqrt{-1}, -1, 1)$.
The first case (with one isolated fixed point and a line of fixed points)
is ruled as in the first part of the proof of Lemma~\ref{lema-3}. 
For the second case, we have three isolated fixed points of order~$4$. 
The line joining two of them, say $P_1,P_2$, is a line of fixed points for $\Phi_4^2$.
The points $P_i$ cannot be singular points of the curve $C$. The only possible option 
is to have an orbit of four singular points and another one of two points. But the orbit of
four points is formed by aligned points and it is forbidden by Bézout Theorem.

In the case $n=5$, let $\Phi_5$ be such a automorphism. If there is a line of fixed points
we conclude again as in the first part of the proof of Lemma~\ref{lema-3}.
Let us assume that there are three isolated fixed points, which cannot be singular in the curve.
The set singular points must be the union of orbits of~$5$ elements, which is not possible.

For the case $n=6$, the restrictions for $n=2,3$ give only one possible case,
corresponding to an automorphism $\Phi_6$ conjugate to a diagonal matrix $(-\zeta,-\overline{\zeta},1)$, hence only
three fixed points which cannot be singular points. The singular points form one orbit;
then for $\Phi_6^3$ we would have three orbits of two points which has been ruled out in Lemma~\ref{lema-2}.
\end{proof}

These curves have interesting properties from the birational point of view. Let $C\in\Lambda_{\text{\rm alg}}$
(though most of the following facts may be also valid in the symplectic case). Le $P_1,\dots,P_6$ be the singular points.
They are not in a conic (from Bézout Theorem). Let $\mathcal{C}_i$, $1\leq i\leq 6$, be the unique conic passing 
through $P_1,\dots,\widehat{P_i},\dots,P_6$. Again, by Bézout Theorem, these conics are irreducible.

\begin{prop}\label{prop:cremona}
Let $\Psi:\bp^2\dashrightarrow\bp^2$ the birational map obtained by blowing-up the points $P_1,\dots,P_6$ and blowing-down
the strict transforms of $\mathcal{C}_1,\dots,\mathcal{C}_6$. Let $Q_1,\dots,Q_6$ be the \emph{images} of the conics and let 
$\mathcal{D}_1,\dots,\mathcal{D}_6$ the \emph{images} of the exceptional components.

Then, the strict transform of $C$ is a smooth quartic curve $D$ passing through the points $Q_1,\dots,Q_6$. There exist
six points $R_1,\dots,R_6$ such that as divisors
\[
D\cdot\mathcal{D}_i= Q_1+\dots+\widehat{Q_i}+\dots+Q_6+3R_i.
\]
These twelve points are pairwise distinct.
\end{prop}

Note that in particular, $C$ is not hyperelliptic.

\begin{proof}
The map $\Psi$ factors as in the following diagram
\[
\begin{tikzcd}
&X\ar[dl, "\sigma_1" above=2pt]\ar[dr,  "\sigma_2" above=2pt]&\\
\bp^2\ar[rr,dashrightarrow, "\Psi"]&&\bp^2
\end{tikzcd}
\]
The map $\sigma_1$ is the composition of the blow-ups of the points $P_1,\dots,P_6$. Under these blow-ups, let us denote
by $\mathcal{D}_i$ the exceptional divisors, and denote also by $\mathcal{C}_i$ the strict transform of $\mathcal{C}_i$.
As each conic has been affected by $5$ blow-ups, $(\mathcal{C}_i)^2_X=-1$, and these strict transforms are pairwise
disjoint. Hence the map $\sigma_2$ is the blow-down of the curves~$\mathcal{C}_i$. Under these blow-downs, the
images $\mathcal{D}_i = \sigma_2(\mathcal{D}_i)$ are conics; they pass through $5$ of the six exceptional
points $Q_j=\sigma_2(\mathcal{C}_j)$.

The other interesection point is the strict transform of a singular point which becomes a smooth point after blowing-up having
intersection number~$3$ with the exceptional divisor.
\end{proof}

Unfortunately, this description is not useful for the computations.

\section{Algebraic curves with \texorpdfstring{$\bz/2$}{Z2}-action}
\label{sec:action2}

From the lemmas in \S\ref{sec:properties}, we can assume that $C_{8,2}\in\Lambda_{\text{\rm alg}}$ is fixed by the involution $\Phi_2:\bp^2\to\bp^2$ given by
$\Phi_2([x:y:z])=[x:y:-z]$ and that two of the $\mathbb{E}_6$ points are $P_1=[1:0:0]$ and $P_2=[0:1:0]$. The 
isolated fixed point $[0:0:1]$ is not in the curve. The tangent lines to $P_i$ must be fixed by the action and 
it is easily seen that they are not tangent to $z=0$, hence the tangent lines are $L_x:\{x=0\}$ and $L_y:\{y=0\}$.

The quotient $\bp^2/\Phi_2$ is isomorphic to the weighted projective plane $\bp^2_\omega$, $\omega=(1,1,2)$, and the map 
is $\pi:\bp^2\to\bp^2_\omega$ where $\pi([x:y:z])=[x:y:z^2]_\omega$. From the orbifold
point of view there is an orbifold $X_2$ constructed on $\bp^2_\omega$ with the usual orbifold
structure around $[0:0:1]_\omega$ and also on the line $L_z:\{z=0\}$, with an action of the cyclic
group of order~$2$.

\begin{lemma}
Let $\tilde{C}_{8,2}:=\Phi_2(C_{8,2})$. Then $\tilde{C}_{8,2}$ is a curve of $\omega$-degree~$8$, 
with two singular points $\mathbb{E}_6$ and two ordinary cusps (not two of them in the same curve of $\omega$-degree~$1$). Moreover
there is a curve of $\omega$-degree~$2$ tangent to the two cusps.
\end{lemma}

This is obvious from the description of $C_{8,2}$. Note that the two cusps come from singular points in the line of fixed points, so they 
are not in a curve of $\omega$-degree~$1$; for any other pair of points, the fact that two singular points are not on a curve of
$\omega$-degree~$1$ follows immediately from~\eqref{eq:Bezout}. 

The way to compute the space of all such curves (up to automorphism) is the following
one. We start with a polynomial 
\[
f(x, y, z) = \sum_{i+j+2k=8} a_{ijk} x^i y^j z^k.
\]
Since it does not pass through $[0:0:1]_\omega$, we may assume that $a_{004}=1$. Recall that 
\[
\aut\bp^2_\omega =
\left\{\Phi_{B,c}\mid
\right.
\left|
B\in\GL(2;\bc),\quad c\in\bc^3
\right\}
\]
where $\Phi_{B,d}([x:y:z]_\omega)=[b_{11} x + b_{12} y:b_{21} x +b_{22} y:z+c_{xx} x^2 +c_{xy} x y +c_{yy} y^2]_\omega
$
for 
\[
B=
\begin{pmatrix}
b_{11}&b_{12}\\
b_{21}&b_{22}
\end{pmatrix},
\quad 
c = (c_{xx}, c_{xy}, c_{yy}).
\]
Note that $\Phi_{B,c}=\Phi_{-B,c}$. 
Using this group we can assume that the two cusps are at 
$[1:0:0]_\omega$ and $[0:1:0]_\omega$ and $z=0$ is
the $2$-curve tangent to the cusps and one of $\mathbb{E}_6$
points is $[1:1:1]_\omega$. The coordinates of the other 
$\mathbb{E}_6$ need to be computed. Note that 
$[x: y: z]_\omega\mapsto[y: x: z]_\omega$ is the only automorphism
fixing this family of curves.

\begin{remark}
Altough the above approach is quite natural, computations become too heavy and they do not end with a solution.
\end{remark}

There is an automorphism of $\bp^2_\omega$ sending the above family of curves to curves 
satisfying:
\begin{itemize}
\item the $\mathbb{E}_6$ points are at 
$[1:0:0]_\omega$ and $[0:1:0]_\omega$;
\item the cusps are at $[1:1:0]_\omega$ and $[a_1:1:1]_\omega$;
\item the $2$-curve tangent to the cusps is $z=b x y$ for some $b$;
\item as with the previous family the are fixed by $[x: y: z]_\omega\mapsto[y: x: z]_\omega$.

\end{itemize}

The conditions about the singular points give a system of equations. Direct attempts failed
and in the notebook \texttt{OcticInvolution} of \texttt{Sagemath} we obtain the existence
of a unique solution up to automorphism. We have normalized this solution to have a simpler form.

\begin{theorem}
Let $C_{8,2}$ be a projective plane curve of degree~$8$ having $6$ singular points of type~$\mathbb{E}_6$
and fixed by an involution. Then it is projectively equivalent to the curve of equation
\begin{align*}
-\frac{11}{3} x^{5} y^{3} - \frac{407}{16} x^{4} y^{4} - 44 x^{3} y^{5} - \frac{11}{8} x^{4} y^{2} z^{2} + \frac{33}{2} x^{2} y^{4} z^{2} + \frac{27}{176} x^{4} z^{4} &\\ - \frac{4}{11} x^{3} y z^{4} 
- \frac{49}{11} x^{2} y^{2} z^{4} - \frac{48}{11} x y^{3} z^{4} + \frac{243}{11} y^{4} z^{4} - \frac{5}{6} x^{2} z^{6} + 10 y^{2} z^{6} + z^{8}&=0
\end{align*}
This curve is not fixed by any other automorphism.
\end{theorem}
The proof of the unicity relies on the \texttt{Sagemath} worksheet, but the fact that this equation satisfies the condition is much easier, see \texttt{CheckCurveInvolution}.

\begin{theorem}
The fundamental group of the complement of $C_{8,2}$ is
\begin{gather*}
G_2 = \langle x, y, z\mid 
[x, z] = 1,\
x y x = y x y,\
y z y = z y z,\
(x y^2 z)^2=1
\rangle,\quad\\
G_2/G_2'\cong\bz/8,\quad
G_2'/G_2''\cong\bz/3,\quad
G_2''/G_2'''\cong(\bz/2)^4\quad
G_2'''\cong\bz^3\times\bz/2.
\end{gather*}
In particular it is not isomorphic to the fundamental group
in Proposition{\rm~\ref{prop:group_symp}}, and hence 
$C_{8,2}$ is not isotopic to the symplectic curve in{\rm~\S\ref{sec:kummer}}.
\end{theorem}

This theorem has been proved using \texttt{Sagemath} and \texttt{Sirocco}, see the details in the notebook \texttt{FundamentalGroupInvolution}. Note that \texttt{Sirocco} uses interval arithmetic which certifies the results.

\section{Algebraic curves with \texorpdfstring{$\bz/3$}{Z3}-action}
\label{sec:accion3}

From the lemmas of \S\ref{sec:properties} we may assume that the automorphism of order~$3$ is
$\Phi_3:\bp^2\to\bp^2$, $\Phi_3([x:y;z]) = [\zeta x:\overline{\zeta} y:z]$ where $\zeta$ is a primitive
cubic root of unity. 
Let $C_{8,3}\in\Lambda_{\text{\rm alg}}$ fixed by the $\Phi_3$.
Let $X_3:=\bp^2/\Phi_3$ its quotient and let
$D_{8,3}\subset X_3$ be the image of $C_{8,3}$. The surface $X_3$ is normal with three isolated cyclic
points of type $\frac{1}{3}(1,-1)$. This notation stands for the following. Let $\mu_d$ be the group of $d$-roots
of unity in $\bc$. Then $\frac{1}{d}(a,b)$ is the quotient of $\bc^2$ by the action of $\mu_d$ defined 
by $\zeta\cdot(x,y)=(\zeta^a x,\zeta^b y)$.

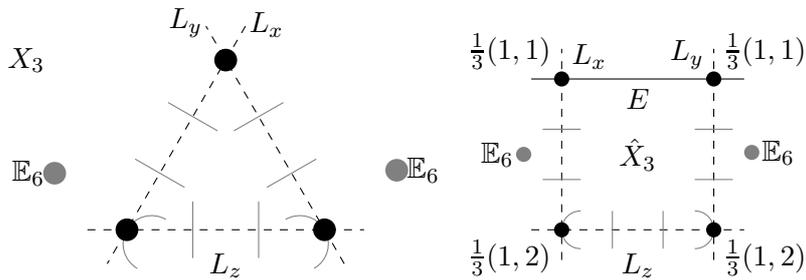
\begin{figure}[ht]
\centering
﻿\begin{tikzpicture}[scale=1.5]
\newcommand\segmento[4]{
\draw[#4] (${(1+#3)}*#1 -#3*#2$) -- (${(1+#3)}*#2 -#3*#1$);}
\foreach \x in {0, 1, 2}
{
\coordinate (A\x) at (90+120*\x:1);

}
\segmento{(A0)}{(A1)}{.2}{dashed}
\segmento{(A1)}{(A2)}{.2}{dashed}
\segmento{(A2)}{(A0)}{.2}{dashed}
\node at (-1.75,1) {$X_3$};

\draw[gray,rotate=-30, shift={(A1)}] (.25,-.25) to [out=180,in=-90](0,0) to[out=90,in=180] (.25,.25);
\draw[gray,rotate=-150, shift={(A2)}] (.25,-.25) to [out=180,in=-90](0,0) to[out=90,in=180] (.25,.25);

\draw[gray,shift={($2/3*(A1)+1/3*(A2)$)}] (0,-.25) -- (0,.25);
\draw[gray,shift={($1/3*(A1)+2/3*(A2)$)}] (0,-.25) -- (0,.25);

\draw[gray,rotate=60, shift={($2/3*(A0)+1/3*(A1)$)}] (0,-.25) -- (0,.25);
\draw[gray,rotate=60, shift={($1/3*(A0)+2/3*(A1)$)}] (0,-.25) -- (0,.25);

\draw[gray,rotate=-60, shift={($2/3*(A0)+1/3*(A2)$)}] (0,-.25) -- (0,.25);
\draw[gray,rotate=-60, shift={($1/3*(A0)+2/3*(A2)$)}] (0,-.25) -- (0,.25);

\fill[gray] (180:1.5)node[left, black] {$\mathbb{E}_6$} circle [radius=.1];
\fill[gray] (1:1.5)node[right, black] {$\mathbb{E}_6$} circle [radius=.1];

\foreach \x in {0, 1, 2}
{
\fill (A\x) circle [radius=.1];
}

\node[above left=5pt] at (A0) {$L_y$};
\node[above right=5pt] at (A0) {$L_x$};
\node[below=5pt] at ($.5*(A1) + .5*(A2)$) {$L_z$};
\end{tikzpicture}
 ﻿\begin{tikzpicture}
\newcommand\segmento[4]{
\draw[#4] (${(1+#3)}*#1 -#3*#2$) -- (${(1+#3)}*#2 -#3*#1$);}
\foreach \x in {0, 1}
\foreach \y in {0, 1}
{
\coordinate (A\x\y) at (2*\x - 1, 2*\y - 1);

}
\segmento{(A00)}{(A01)}{.2}{dashed}
\segmento{(A10)}{(A11)}{.2}{dashed}
\segmento{(A01)}{(A11)}{.2}{}
\segmento{(A00)}{(A10)}{.2}{dashed}
\node at (0,0) {$\hat{X}_3$};

\draw[gray,rotate=0, shift={(A00)}] (.25,-.25) to [out=180,in=-90](0,0) to[out=90,in=180] (.25,.25);

\draw[gray,rotate=180, shift={(A10)}] (.25,-.25) to [out=180,in=-90](0,0) to[out=90,in=180] (.25,.25);

\draw[gray,shift={($2/3*(A00)+1/3*(A10)$)}] (0,-.25) -- (0,.25);

\draw[gray,shift={($1/3*(A00)+2/3*(A10)$)}] (0,-.25) -- (0,.25);

\draw[gray,rotate=90, shift={($2/3*(A00)+1/3*(A01)$)}] (0,-.25) -- (0,.25);

\draw[gray,rotate=90, shift={($1/3*(A00)+2/3*(A01)$)}] (0,-.25) -- (0,.25);

\draw[gray,rotate=90, shift={($2/3*(A10)+1/3*(A11)$)}] (0,-.25) -- (0,.25);

\draw[gray,rotate=90, shift={($1/3*(A10)+2/3*(A11)$)}] (0,-.25) -- (0,.25);

\fill[gray] (180:1.5)node[left, black] {$\mathbb{E}_6$} circle [radius=.1];
\fill[gray] (1:1.5)node[right, black] {$\mathbb{E}_6$} circle [radius=.1];

\foreach \x in {0, 1}
\foreach \y in {0, 1}
{
\fill (A\x\y) circle [radius=.1];
}
\node[below left] at (A00) {$\frac{1}{3}(1, 2)$};
\node[below right] at (A10) {$\frac{1}{3}(1, 2)$};
\node[above left] at (A01) {$\frac{1}{3}(1, 1)$};
\node[above right] at (A11) {$\frac{1}{3}(1, 1)$};
\node[above right] at (A01) {$L_x$};
\node[above left] at (A11) {$L_y$};
\node[below=5pt] at ($.5*(A10) + .5*(A00)$) {$L_z$};
\node[below] at ($.5*(A11) + .5*(A01)$) {$E$};
\end{tikzpicture}
 \caption{Surface $X_3$ with the image of the curve and $(1,1)-$blow-up of \ref{B1}.}
\label{fig:cociente3}
\end{figure}

There is a birational transformation to pass from $X_3$
to $\bp^2$. These are the steps:
\begin{enumerate}[label=(B\arabic{enumi})]
\item\label{B1} $(1,1)$-blow-up of the image of $[0:0:1]$ in $X_3$, with exceptional component~$E$. We obtain a \emph{singular ruled surface}
with four singular points in two fibers, the strict transforms
of $L_x,L_y$. The new ones are of type $\frac{1}{3}(1,1)$. The two sections in the right-hand side
of Figure~\ref{fig:cociente3} have self-interesection $\frac{1}{3}$ ($L_z$ below) and $-\frac{1}{3}$ ($E$ above),
see~\cite{AMO:2014a} for details on weighted blow-ups.

\item\label{B2} $(1,1)$-blow-up of the two points of type $\frac{1}{3}(1,1)$, with exceptional
components $E_x, E_y$ of self-interesection~$-3$. The self-interesection
of the strict transforms of $L_x,L_y$ is~$-\frac{1}{3}$.

\item\label{B3} Blow-down of the strict transforms of the images of the lines $x=0$, $y=0$; it is the inverse of a $(1,3)$-blow-up of a smooth points. The result is a smooth surface,
actually the Hirzebruch ruled surface $\Sigma_1$ where the $(-1)$-curve is $E$.

\item\label{B4} Contract the $(-1)$-curve $E$.

\end{enumerate}

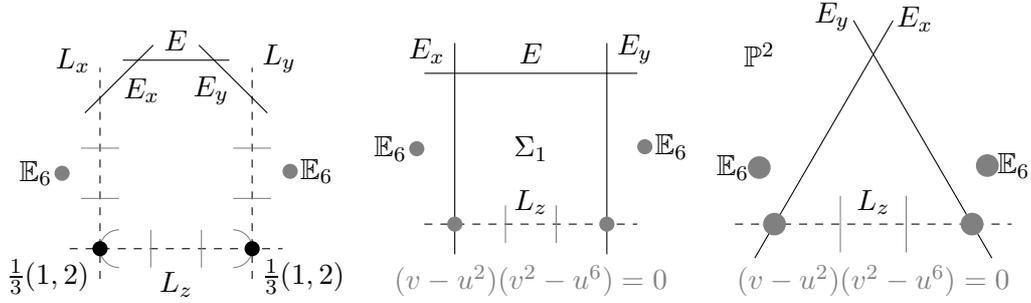
\begin{figure}[ht]
\centering
﻿\begin{tikzpicture}
\newcommand\segmento[4]{
\draw[#4] (${(1+#3)}*#1 -#3*#2$) -- (${(1+#3)}*#2 -#3*#1$);}
\foreach \x in {0, 1}
\foreach \y in {0, 1}
{
\coordinate (A\x\y) at (2*\x - 1, 2*\y - 1);
}

\coordinate (B0) at (-1/2, 3/2) {};
\coordinate (B1) at (1/2, 3/2) {};

\segmento{(A00)}{(A01)}{.2}{dashed}
\segmento{(A10)}{(A11)}{.2}{dashed}
\segmento{(B0)}{(B1)}{.2}{}
\segmento{(B0)}{(A01)}{.4}{}
\segmento{(B1)}{(A11)}{.4}{}
\segmento{(A00)}{(A10)}{.2}{dashed}

\draw[gray,rotate=0, shift={(A00)}] (.25,-.25) to [out=180,in=-90](0,0) to[out=90,in=180] (.25,.25);

\draw[gray,rotate=180, shift={(A10)}] (.25,-.25) to [out=180,in=-90](0,0) to[out=90,in=180] (.25,.25);

\draw[gray,shift={($2/3*(A00)+1/3*(A10)$)}] (0,-.25) -- (0,.25);

\draw[gray,shift={($1/3*(A00)+2/3*(A10)$)}] (0,-.25) -- (0,.25);

\draw[gray,rotate=90, shift={($2/3*(A00)+1/3*(A01)$)}] (0,-.25) -- (0,.25);

\draw[gray,rotate=90, shift={($1/3*(A00)+2/3*(A01)$)}] (0,-.25) -- (0,.25);

\draw[gray,rotate=90, shift={($2/3*(A10)+1/3*(A11)$)}] (0,-.25) -- (0,.25);

\draw[gray,rotate=90, shift={($1/3*(A10)+2/3*(A11)$)}] (0,-.25) -- (0,.25);

\fill[gray] (180:1.5)node[left, black] {$\mathbb{E}_6$} circle [radius=.1];
\fill[gray] (1:1.5)node[right, black] {$\mathbb{E}_6$} circle [radius=.1];

\foreach \x in {0, 1}
{
\fill (A\x0) circle [radius=.1];
}
\node[below left] at (A00) {$\frac{1}{3}(1, 2)$};
\node[below right] at (A10) {$\frac{1}{3}(1, 2)$};
\node[above left] at ($1.1*(A01)-.1*(A00)$) {$L_x$};
\node[above right] at ($1.1*(A11)-.1*(A10)$) {$L_y$};
\node[right=5pt] at ($1.05*(A01)-.05*(A00)$) {$E_x$};
\node[left=5pt] at ($1.05*(A11)-.05*(A10)$) {$E_y$};
\node[below=5pt] at ($.5*(A10) + .5*(A00)$) {$L_z$};
\node[above] at ($.5*(B0) + .5*(B1)$) {$E$};

\end{tikzpicture}
 ﻿\begin{tikzpicture}
\newcommand\segmento[4]{
\draw[#4] (${(1+#3)}*#1 -#3*#2$) -- (${(1+#3)}*#2 -#3*#1$);}
\foreach \x in {0, 1}
\foreach \y in {0, 1}
{
\coordinate (A\x\y) at (2*\x - 1, 2*\y - 1);

}
\segmento{(A00)}{(A01)}{.2}{}
\segmento{(A10)}{(A11)}{.2}{}
\segmento{(A01)}{(A11)}{.2}{}
\segmento{(A00)}{(A10)}{.2}{dashed}
\node at (0,0) {$\Sigma_1$};

\draw[gray,shift={($2/3*(A00)+1/3*(A10)$)}] (0,-.25) -- (0,.25);

\draw[gray,shift={($1/3*(A00)+2/3*(A10)$)}] (0,-.25) -- (0,.25);

\fill[gray] (180:1.5)node[left, black] {$\mathbb{E}_6$} circle [radius=.1];
\fill[gray] (1:1.5)node[right, black] {$\mathbb{E}_6$} circle [radius=.1];

\foreach \x in {0, 1}
{
\fill[gray] (A\x0) circle [radius=.1];
}
\node[gray] at (0,-1.75) {$(v-u^2)(v^2-u^6)=0$};
\node[above left] at (A01) {$E_x$};
\node[above right] at (A11) {$E_y$};
\node[above] at ($.5*(A10) + .5*(A00)$) {$L_z$};
\node[above] at ($.5*(A11) + .5*(A01)$) {$E$};
\end{tikzpicture}
 ﻿\begin{tikzpicture}[scale=1.5]
\newcommand\segmento[4]{
\draw[#4] (${(1+#3)}*#1 -#3*#2$) -- (${(1+#3)}*#2 -#3*#1$);}
\foreach \x in {0, 1, 2}
{
\coordinate (A\x) at (90+120*\x:1);

}
\segmento{(A0)}{(A1)}{.2}{}
\segmento{(A1)}{(A2)}{.2}{dashed}
\segmento{(A2)}{(A0)}{.2}{}
\node at (-1,1) {$\mathbb{P}^2$};

\draw[gray,shift={($2/3*(A1)+1/3*(A2)$)}] (0,-.25) -- (0,.25);
\draw[gray,shift={($1/3*(A1)+2/3*(A2)$)}] (0,-.25) -- (0,.25);

\fill[gray] (180:1)node[left, black] {$\mathbb{E}_6$} circle [radius=.1];
\fill[gray] (1:1)node[right, black] {$\mathbb{E}_6$} circle [radius=.1];

\foreach \x in {1, 2}
{
\fill[gray] (A\x) circle [radius=.1];
}
\node[gray] at (0,-1) {$(v-u^2)(v^2-u^6)=0$};

\node[above left=5pt] at (A0) {$E_y$};
\node[above right=5pt] at (A0) {$E_x$};
\node[above] at ($.5*(A1) + .5*(A2)$) {$L_z$};
\end{tikzpicture}
 \caption{$(1,1)$-Blow-ups of \ref{B2}, blow-downs of \ref{B3} and blow-down of \ref{B4}.}
\label{fig:cociente3-1}
\end{figure}

Actually all this operation has simple coordinates. The composition of the quotient 
and the birational map is a rational map $\Theta:\bp^2\dashrightarrow\bp^2$ given by 
\[
\Theta([x:y:z])=[x^3:y^3:x y z].
\]

\begin{lemma}\label{lema:auto3}
Let $C$ the image of $C_{8,3}$ by $\Theta$. Then $C$ is a curve of degree~$8$, 
with two singular points $\mathbb{E}_6$ and two singularities with the topological type of $(u-v^2)(u^2-v^6)=0$ having
maximal contact with the tangent line.
\end{lemma}

We proceed as in \S\ref{sec:action2}. Let us take 
\[
f(x, y, z) = \sum_{i+j+k=8} a_{ijk} x^i y^j z^k.
\]
with $a_{008}=1$ since $[0:0:1]\notin C$. We place the reducible singular points at 
$[1:0:0]$ and $[0:1:0]$ with respective tangent lines $y=0$ and $z=0$.
One of the $\mathbb{E}_6$ points is at $[1:1:1]$ and for the other one we use two new variables.
The only automorphism
fixing this family of curves is
$[x: y: z]\mapsto[y: x: z]$.

The system of equations is more complicated that the one in \S\ref{sec:action2}, but we managed to obtain the solutions
using \texttt{Sagemath}, see the notebook \texttt{OcticAuto3}. In Appendix~\ref{app:strategy}, the common procedure is explained.
To describe the solution we need to introduce the number field $\mathbb{K}:=\bq[\eta]$, where
$\eta$ is a solution of $p(t):=t^4 - 2 t^3 + t^2 - 2 t - 2$.
This polynomial has two real roots $\eta_1,\eta_2$ and two complex conjugate roots
$\eta_3,\eta_4$.

\begin{theorem}
Let $C_{8,3}$ be a projective plane curve of degree~$8$ having $6$ singular points of type~$\mathbb{E}_6$
and fixed by an automorphism of order~$3$. Then it is projectively equivalent to a curve $C^{\eta_i}_{8,3}$
whose equation 
is obtained as follows. Let
\[
G_0(x,y,z) = \frac{F(x^3, y^3, x y z)}{x^8 y^8},
\]
where $F$ is the equation in the Appendix{\rm~\ref{ap:eq}}. Then, 
$G(x, y, z):= G_0(x+\zeta y, x + \overline{\zeta}y, z)$ with coefficients 
in $\mathbb{K}=\bq[\eta_i]$. This curve is not fixed by any other automorphism.

The fundamental group of the complement of any such curve is cyclic 
of order~$8$.
\end{theorem}
The proof of this theorem can be checked 
in \texttt{OcticAuto3}. The computation of the fundamental
group takes much longer than it took in the case of \S\ref{sec:action2}
and it has been done with \texttt{Sagemath} and \texttt{Sirocco},
see the notebook \texttt{FundamentalGroupAuto3}.

As for the other type of curves, the long computation is only needed
to prove that these curves are the only ones. It is easier to prove that 
the satisfy the required condition, see \texttt{CheckCurveAuto3}.

\section{Alternative way to compute the fundamental groups}

There is an alternative way to compute this fundamental group. 
We can compute $G_3^{\textrm{orb}}:=\pi_1^{\textrm{orb}}(X_3\setminus D_{8,3}^{\eta_i})$
and  $G_2^{\textrm{orb}}:=\pi_1^{\textrm{orb}}(X_2\setminus\tilde{C}_{8,2})$. In this particular situation
it does not really save computation time but in other cases it allows to obtain a faster and 
computer-free approach.

The orbifold fundamental group $\pi_1^{\textrm{orb}}(X_2\setminus\tilde{C}_{8,2})$ is computed 
following several steps, see \texttt{Alternatives2}:
\begin{enumerate}[label=(Orb$^2$\arabic{enumi})]
\item Blow up $[0:0:1]_\omega$; we obtain a surface $\Sigma_2$ (a ruled Hirzebruch surface) with an
exceptional component $E$, with self-interesection $-2$. We compute
the group $\pi_1(\Sigma_2\setminus(\tilde{C}_{8,2}\cup L_z\cup E))$.

\item To compute this group we consider an affine chart, say the complement of $E$ and $L_x$, using 
the standard Zariski-van Kampen method. In \texttt{Alternatives2} we have a finitely presented
group with five generators $x_0,\dots,x_4$, where $x_2$ is a meridian of $L_z$ and 
$e:=(x_0\cdot\ldots\cdot x_4)^{-1}$ is a meridian of~$E$. Following \cite{KhKu:03}, a meridian
of $L_x$ is $e^2$.

\item The group $G^2_{\textrm{orb}}$ is obtained by adding the relations $x_2^2=e^2=1$.

\item The group $G_2$ is the kernel of the map $G_2^{\textrm{orb}}\twoheadrightarrow\mathbb{Z}/2$
defined by $x_i\mapsto 0$, $i\neq 2$, and $x_2\mapsto 1$. In \texttt{Alternatives2}, we prove 
that $x_2$ is central and of order~$2$. Hence $G_2^{\textrm{orb}}\cong G_2\times\bz/2$.

\item Actually $G_2$ is the orbifold fundamental group of the complement of $\tilde{C}_{8,2}$, where 
the unique orbifold point is the singular one.
\end{enumerate}

We follow a similar strategy to compute the orbifold fundamental groups $G_3^\textrm{orb}=\pi_1^{\textrm{orb}}(X_3\setminus\tilde{D}_{8,3}^{\eta_i})$, see \texttt{Alternatives3}:
\begin{enumerate}[label=(Orb$^3$\arabic{enumi})]
\item We start with the final birational model of the rational map and compute
$\pi_1(\bp^2\setminus(\tilde{D}_{8,3}^{\eta_i}\cup E_x\cup E_y))$. Actually, we take the affine chart
of the complement of $E_x$ and compute the fundamental group of the complement of 
$\tilde{D}_{8,3}^{\eta_i}$ and $E_y$.

\item Using the standard Zariski-van Kampen method we obtain in \texttt{Alternatives3} a finitely presented
group with nine generators $x_0,\dots,x_8,\ell_y$, where $e_y$ is a meridian of $E_y$, 
the $x_i$'s are meridians of $\tilde{D}_{8,3}^{\eta_i}$ 
$e:=(x_0\cdot\ldots\cdot x_8)^{-1}$ is a meridian of~$E$, and $e_x:=e_y^{-1}\cdot e$ is a meridian
of $E_x$.

\item Following \cite{Mumford1961}, we deduce that for the group $G_3^\textrm{orb}$
we have to add the relations deduced from the divisor $E+E_x+E_y$ in Figure~\ref{fig:smooth}:
\[
e_x\cdot e_y = e\ (\text{known}),\quad e=e_x^3=e_y^3
\Rightarrow e=e_x\cdot e_y=e_x^3=1.
\]

\item With this new relation we have computed in \texttt{Alternatives3} that all the groups
are $\bz/24$ and hence we recover the abelianity of $G_3$.

\end{enumerate}

\begin{figure}[ht]
\centering
﻿\begin{tikzpicture}
\newcommand\segmento[4]{
\draw[#4] (${(1+#3)}*#1 -#3*#2$) -- (${(1+#3)}*#2 -#3*#1$);}
\foreach \x in {0, 1}
\foreach \y in {0, 1}
{
\coordinate (A\x\y) at (4*\x - 2, 4*\y - 1);
}

\coordinate (B0) at (-3/2, 7/2) {};
\coordinate (B1) at (3/2, 7/2) {};

\coordinate (A0) at ($.8*(A00) + .2*(A10)$);
\coordinate (A1) at ($.2*(A00) + .8*(A10)$);
\coordinate (C0) at ($.7*(A00) + .3*(A01)$);
\coordinate (C1) at ($.7*(A10) + .3*(A11)$);
\coordinate (D0) at ($.2*(A00) + .3*(A0) + .3*(C0)$);
\coordinate (D1) at ($.2*(A10) + .3*(A1) + .3*(C1)$);

\segmento{(C0)}{(D0)}{.4}{}
\segmento{(C1)}{(D1)}{.4}{}
\segmento{(A0)}{(D0)}{.4}{}
\segmento{(A1)}{(D1)}{.4}{}
\segmento{(A00)}{($1.2*(A01)-.2*(A00)$)}{-.1}{dashed}
\segmento{(A10)}{($1.2*(A11)-.2*(A10)$)}{-.1}{dashed}
\segmento{(B0)}{(B1)}{.1}{}
\segmento{(B0)}{(A01)}{.4}{}
\segmento{(B1)}{(A11)}{.4}{}
\segmento{(A00)}{(A10)}{-.1}{dashed}

\draw[gray,rotate=-45,shift={($.5*(C0)+.5*(D0)$)}] (0,-.25) -- (0,.25);

\draw[gray,rotate=45,shift={($.5*(C1)+.5*(D1)$)}] (0,-.25) -- (0,.25);

\draw[gray,shift={($2/3*(A00)+1/3*(A10)$)}] (0,-.25) -- (0,.25);

\draw[gray,shift={($1/3*(A00)+2/3*(A10)$)}] (0,-.25) -- (0,.25);

\draw[gray,rotate=90, shift={($.5*(A00)+.5*(A01)$)}] (0,-.25) -- (0,.25);

\draw[gray,rotate=90, shift={($1/4*(A00)+3/4*(A01)$)}] (0,-.25) -- (0,.25);

\draw[gray,rotate=90, shift={($.5*(A10)+.5*(A11)$)}] (0,-.25) -- (0,.25);

\draw[gray,rotate=90, shift={($1/4*(A10)+3/4*(A11)$)}] (0,-.25) -- (0,.25);

\fill[gray] (180:1)node[right, black] {$\mathbb{E}_6$} circle [radius=.1];
\fill[gray] (0:1)node[left, black] {$\mathbb{E}_6$} circle [radius=.1];

\node[above right] at ($.55*(A01)+.45*(A00)$) {$L_x(-1)$};
\node[above left] at ($.55*(A11)+.45*(A10)$) {$(-1)L_y$};
\node[right=5pt] at ($1.05*(A01)-.05*(A00)$) {$E_x(-3)$};
\node[left=5pt] at ($1.05*(A11)-.05*(A10)$) {$(-3)E_y$};
\node[below=0pt] at ($.5*(A10) + .5*(A00)$) {$L_z(-1)$};
\node[above] at ($.5*(B0) + .5*(B1)$) {$E(-1)$};

\node[left=3pt] at ($(C0)$) {$-2$};
\node[right=3pt] at ($(C1)$) {$-2$};

\node[below left] at ($(A0)$) {$-2$};
\node[below right] at ($(A1)$) {$-2$};

\end{tikzpicture}
 \caption{Minimal resolution of $\hat{X}_3$}
\label{fig:smooth}
\end{figure}
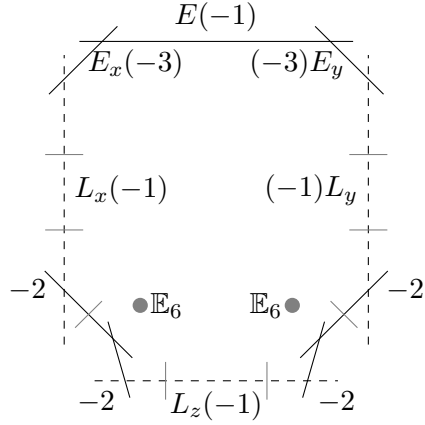

In \texttt{Alternatives3} we have computed a simplified braid monodromy for 
$\tilde{D}_{8,3}^{\eta_i}$ which may give some hints about the topological
equivalence of these curves.

\section{Conclusions}

We summarize the results an open questions.

\begin{enumerate}[label=(C\arabic{enumi})]
\item There is no homeomorphism $\Phi_i:\bp^2\to\bp^2$
such that $\Phi_i(C_{8,3}^{\eta_i})=C_{8,2}$.

\item There is no homeomorphism $\Psi_i:\bp^2\to\bp^2$
such that $\Psi_i(C_{8,3}^{\eta_i})=C_{\text{\rm symp}}$.

\item There is no homeomorphism $\Psi_i:\bp^2\to\bp^2$
such that $\Psi_i(C_{8,2})=C_{\text{\rm symp}}$.

\item The complex conjugation is a homeomorphism $\Phi:\bp^2\to\bp^2$
such that $\Phi(C_{8,3}^{\eta_3})=C_{8,3}^{\eta_4}$.

\item The existence of homeomorphisms
$\Phi_{i,j}:\bp^2\to\bp^2$
such that $\Phi(C_{8,3}^{\eta_i})=C_{8,3}^{\eta_j}$
is an open question, for $i\neq j$ and $\{i,j\}\neq\{3, 4\}$.

\item The existence of a homeomorphism
$\Phi_{3,4}:\bp^2\to\bp^2$
such that $\Phi(C_{8,3}^{\eta_3})=C_{8,3}^{\eta_4}$ which
preserves the orientation of the curves
is an open question.

\item The existence of other curves in $\Lambda_{\text{\rm alg}}$
is an open question.

\item The existence of curves in $\Lambda_{\text{\rm alg}}$
isotopic to $C_{\text{\rm symp}}$
is an open question.

\end{enumerate}

\section{Perspectives}

A direct approach to compute $\Lambda_{\text{\rm alg}}$ seems to be hopeless.
Isolating special properties for the known solutions would help  to get new
ideas that would allow either to discard new cases or to obtain some new ones.
We know that there is no more curve in $\Lambda_{\text{\rm alg}}$ fixed by 
a non-trivial homeomorphism.

In particular we are going to compute the smooth quartics
of \S\ref{sec:properties}. 
Let us consider the the $2$-dimensional projective system
formed by the closure of the family
of quintics having ordinary double points at the six points
$P_1,\dots,P_6$.
The intersection number of two such quintics at the base points
is at least~$24$, so, they intersect at another point.
If we blow up the six points the strict transforms of the quintics
are smooth rational curves with self-intersection~$-1$.

In this closure we find also the curves formed by $\mathcal{C}_i$
and a line passing through~$P_i$, which are \emph{exceptional}
elements of the family. Their strict transforms are disjoint
from the strict transforms of the irreducible quintics in the system.
The Cremona transformation described in Proposition~\ref{prop:cremona} is obtained 
by blowing-down the strict transforms of these conics.

The notebooks \texttt{Birational2} and \texttt{Birational3} contain the computations
leading to the following results. Moreover the systems
of points described in Proposition~\ref{prop:cremona} are also computed.

\begin{theorem}
The curve $C_{8,2}$ is birationally equivalent to 
\[
z^{4} - 3 x^{2} z^{2} + y^{2} z^{2} -36 x^{3} y + 45 x^{2} y^{2} - 12 x y^{3}  
=0.
\]
\end{theorem}

\begin{theorem}
The curve $C_{8, 3}^{\eta_i}$ is birationally equivalent to 
\[
z^{4} + \frac{3}{38}b_{12} x y z^2 + \frac{1}{19}(2 b_{01} + \zeta c_{01}) x^3 z + \frac{1}{19}(2b_{01} + \overline{\zeta} c_{01}) y^3 z +\frac{3}{19} b_{20} x^2 y^2
=0,
\]
where
\begin{align*}
b_{12} &= -97 \eta_i^{3} - 23 \eta_i^{2} - 130 \eta_i - 92\\
b_{01} &= 74 \eta_i^{3} + 6 \eta_i^{2} + 109 \eta_i + 75\\
c_{01} &= -51 \eta_i^{3} + \eta_i^{2} - 42 \eta_i - 35\\
b_{20} &= 3596 \eta_i^{3} + 585 \eta_i^{2} + 4862 \eta_i + 3325.
\end{align*}
\end{theorem}
\appendix
\section{Strategy of the computations}\label{app:strategy}

In \S\ref{sec:action2} and \S\ref{sec:accion3} we need to find the zero locus 
of an ideal $J_0$ in a ring $\mathbb{C}[a_1,\dots,a_n]$. More precisely we 
look for \emph{non-degenerate} solutions, since the conditions imposed 
are closed conditions and the space we are looking for is only locally-closed.

The existence of degenerate solutions is a big computational problem. The strategy
followed consists to define a \emph{tree} of ideals whose root is $J_0$. This tree
has levels and at each level we eliminate a variable.

Let us assume that we have inductively constructed un ideal $J_{j,k}\subset\mathbb{C}[a_1,\dots,a_{n-k}]$.
Using heuristic arguments we choose a generator $f_0$ of the ideal and a variable, say $a_{n-k}$, and we compute the resultants with respect to $a_{n-k}$ of $f_0$ with the other generators.
We factorize each one of these resultants and we eliminate the factors which are known to provide 
degenerate solutions. With the remaining factors, we combine them to give a family of ideals
$J_{j',k +1}$ in $\mathbb{C}[a_1,\dots,a_{n-k}]$.

Some of the leaves of this tree will stop with no solution and we pay attention to the ones 
ending in prime ideals of $\mathbb{C}[a_1]$.

Fix one of these leaves. Actually these ideals have coefficients in $\bq$. 
A prime ideal $J_{i',n-1}\subset\mathbb{C}[a_1]$ determines an extension $\mathbb{L}_1$
of $\bq$ where a solution leaves. Replacing the value of $a_1$ by this solution
in the ideal $J_{i,n-2}$ we obtain a new ideal in $\mathbb{L}_1[a_2]$. We factorize
these principal ideals. Either some of these process stop with no solution or 
we end with one solution.

In both cases we end with only one algebraic solution. For the case of \S\ref{sec:action2}
the solution lives in a degree~$2$ extension of $\bq$ but the symmetry allows us
to end with a rational solution. In the case of \S\ref{sec:accion3}
the solution lives in $\mathbb{K}_1=\bq[\eta,\zeta]$, extension of $\bq$ 
of degree~$8$. The symmetry allows us
to end with a solution
$\mathbb{K}_1=\bq[\eta]$, extension of $\bq$ 
of degree~$4$.

\section{Equations}\label{ap:eq}

Let $\mathbb{K}_1:=\bq[\eta, \zeta]$ and let $\sigma$ be the non-trivial automorphism of $\mathbb{K}_1$; the field 
$\bq[\eta]$ is the fixed field by $\sigma$. When $\eta$ is real, $\sigma$ is the complex conjugation.
A curve of Lemma~\ref{lema:auto3}
is of the form 
\[
F(x,y,z) = F_0(x y,z) + 2x y z (x F_1(x y,z) + y F_1^{\sigma}(x y,z))
+ x^2 y^2 (x^2 F_2(x y,z) +  y^2 F_2^{\sigma}(x y,z)),
\]
where $F_0, F_1, F_2\in\mathbb{K}[t,z]$. We have 
\begin{align*}
F_0(t, z) =& 
z^{8} + \frac{2r_{16}}{19} t z^{6} + 
+ \frac{3r_{24}}{19^2} t^{2} z^{4} +
\frac{2r_{32}}{19^2} t^{3} z^{2} +
\frac{4r_{40}}{19} t^{4},
\end{align*}
\begin{align*}
r_{16} &=
437 \eta^{3} - 1270 \eta^{2} + 1130 \eta - 1696
\\
r_{24} &= -596956 \eta^{3} + 1619007 \eta^{2} - 1523682 \eta + 2184414 
\\
s_{23} &= -2064411 \eta^{3} + 5739587 \eta^{2} - 5326476 \eta + 7777170
\\
s_{40} &= 11524593 \eta^{3} - 28834395 \eta^{2} + 28396048 \eta - 38303610.
\end{align*}
\begin{align*}
F_1(t, z) =& 
\frac{r_{15} +\zeta s_{15}}{19^3} z^4 + 
\frac{r_{23}  + \zeta s_{23}}{19^2} t z^{2} +
6\frac{r_{31}  + 4\zeta s_{31}}{19^2} t^{2},\\
\end{align*}
\begin{align*}
r_{15} &= 
-157924 \eta^{3} + 308331 \eta^{2} - 356378 \eta + 387894
\\
s_{15} &= 
182695 \eta^{3} - 547611 \eta^{2} + 485700 \eta - 752178
\\
r_{23} &= 1276065 \eta^{3} - 3104444 \eta^{2} + 3107094 \eta - 4100620 
\\
s_{23} &= -2064411 \eta^{3} + 5739587 \eta^{2} - 5326476 \eta + 7777170
\\
r_{31} &= -5295773 \eta^{3} + 14400235 \eta^{2} - 13528408 \eta + 19435018 
\\
s_{31} &= 6353433 \eta^{3} - 16472958 \eta^{2} + 15895154 \eta - 22035984 
\end{align*}
\begin{align*}
F_2(t, z) =& 
\frac{r_{22}  + 2\zeta s_{22}}{19}z^2 + 
2\frac{r_{30} + 4 \zeta s_{30}}{19} t.
\end{align*}
\begin{align*}
r_{22} &= 74354 \eta^{3} - 196839 \eta^{2} + 187718 \eta - 264358 
\\
s_{22} &= 138989 \eta^{3} - 356263 \eta^{2} + 346016 \eta - 475522 
\\
r_{30} &= -8288405 \eta^{3} + 21480135 \eta^{2} - 20732048 \eta + 28731618 
\\
s_{30} &= -2845567 \eta^{3} + 7360179 \eta^{2} - 7111716 \eta + 9841218 
\end{align*}


\begin{thebibliography}{ACM20b}

\bibitem[AC98]{AC:1998}
E.~Artal and J.~Carmona, \emph{Zariski pairs, fundamental groups and
  {A}lexander polynomials}, J. Math. Soc. Japan \textbf{50} (1998), no.~3,
  521--543.

\bibitem[ACCT01]{ACCT}
E.~Artal, J.~Carmona, J.I. Cogolludo, and H.~Tokunaga, \emph{Sextics with
  singular points in special position}, J. Knot Theory Ramifications
  \textbf{10} (2001), no.~4, 547--578.

\bibitem[ACM20a]{ACM-Nemethi}
E.~Artal, J.I. Cogolludo, and J.~{Martín-Morales}, \emph{Cremona
  transformations of weighted projective planes, {Z}ariski pairs, and rational
  cuspidal curves}, Singularities and Their Interaction with Geometry and Low
  Dimensional Topology (J.~Fernández~de Bobadilla, T.~Laszlo, and A.~Stipsicz,
  eds.), Trends in Mathematics, Birkhäuser, Basel, 2020.

\bibitem[ACM20b]{ACM:2020b}
\bysame, \emph{Triangular curves and cyclotomic {Z}ariski tuples}, Collect.
  Math. \textbf{71} (2020), no.~3, 427--441.

\bibitem[ACM20c]{ACMat:2020}
E.~Artal, J.I Cogolludo, and D.~Matei, \emph{Characteristic varieties of graph
  manifolds and quasi-projectivity of fundamental groups of algebraic links},
  Eur. J. Math. \textbf{6} (2020), no.~3, 624--645. \MR{4151713}

\bibitem[ADT10]{ADT:10}
M.~Amram, M.~Dettweiler, and M.~Teicher, \emph{On rigid covers associated to
  the three-cuspidal quartic}, Abh. Math. Semin. Univ. Hambg. \textbf{80}
  (2010), no.~1, 1--8.

\bibitem[AMO14]{AMO:2014a}
E.~Artal, J.~{Martín-Morales}, and J.~Ortigas, \emph{Intersection theory on
  abelian-quotient {$V$}-surfaces and {$\bf Q$}-resolutions}, J. Singul.
  \textbf{8} (2014), 11--30.

\bibitem[CW05]{CatWaj:05}
F.~Catanese and B.~Wajnryb, \emph{The 3-cuspidal quartic and braid monodromy of
  degree 4 coverings}, Projective varieties with unexpected properties, Walter
  de Gruyter, Berlin, 2005, pp.~113--129.

\bibitem[Dol82]{Dolgachev1982}
I.~Dolgachev, \emph{Weighted projective varieties}, Group actions and vector
  fields ({V}ancouver, {B}.{C}., 1981), Lecture Notes in Math., vol. 956,
  Springer, Berlin, 1982, pp.~34--71.

\bibitem[Fuj82]{Fujita1982}
T.~Fujita, \emph{On the topology of noncomplete algebraic surfaces}, J. Fac.
  Sci. Univ. Tokyo Sect. IA Math. \textbf{29} (1982), no.~3, 503--566.

\bibitem[GAP22]{GAP}
The GAP~Group, \emph{{GAP -- Groups, Algorithms, and Programming, Version
  4.12.2}}, 2022, available at \verb+(http://www.gap-system.org)+.

\bibitem[GS22]{Golla-Starkston:22}
M.~Golla and L.~Starkston, \emph{The symplectic isotopy problem for rational
  cuspidal curves}, Compos. Math. \textbf{158} (2022), no.~7, 1595--1682.

\bibitem[KK03]{KhKu:03}
V.M. Kharlamov and Vik.S. Kulikov, \emph{On braid monodromy factorizations},
  Izv. Ross. Akad. Nauk Ser. Mat. \textbf{67} (2003), no.~3, 79--118.

\bibitem[Kul07]{Kulikov:07}
Vik.S. Kulikov, \emph{Hurwitz curves}, Uspekhi Mat. Nauk \textbf{62} (2007),
  no.~6(378), 3--86.

\bibitem[MR16]{Marco2016}
M.~Marco and M.~Rodríguez, \emph{{\emph{\texttt{SIROCCO}}}: A library for
  certified polynomial root continuation}, Mathematical Software - ICMS 2016,
  Lecture Notes in Comput. Sci., vol. 9725, Springer-Verlag, Berlin, 2016,
  pp.~191--197.

\bibitem[Mum61]{Mumford1961}
D.~Mumford, \emph{The topology of normal singularities of an algebraic surface
  and a criterion for simplicity}, Inst. Hautes \'{E}tudes Sci. Publ. Math.
  (1961), no.~9, 5--22.

\bibitem[Ore00]{Orevkov2000}
S.Yu. Orevkov, \emph{Markov moves for quasipositive braids}, C. R. Acad. Sci.
  Paris S\'{e}r. I Math. \textbf{331} (2000), no.~7, 557--562.

\bibitem[PJe18]{binder}
{P}roject {J}upyter and \emph{et al.}, \emph{{B}inder 2.0 - {R}eproducible,
  interactive, sharable environments for science at scale}, {P}roceedings of
  the 17th {P}ython in {S}cience {C}onference ({F}atih {A}kici, {D}avid
  {L}ippa, {D}illon {N}iederhut, and {M} {P}acer, eds.), 2018, pp.~113 -- 120.

\bibitem[Ste23]{Sagemath}
W.A. et~al. Stein, \emph{Sage {M}athematics {S}oftware ({V}ersion 10.1)}, The
  Sage Development Team, 2023, {\tt http://www.sagemath.org}.

\bibitem[Ulu01]{Uludag:2001}
A.M. Uluda{\u{g}}, \emph{More {Z}ariski pairs and finite fundamental groups of
  curve complements}, Manuscripta Math. \textbf{106} (2001), no.~3, 271--277.

\end{thebibliography}
\providecommand{\bysame}{\leavevmode\hbox to3em{\hrulefill}\thinspace}
\providecommand{\MR}{\relax\ifhmode\unskip\space\fi MR }
\providecommand{\MRhref}[2]{%
  \href{http://www.ams.org/mathscinet-getitem?mr=#1}{#2}
}
\providecommand{\href}[2]{#2}

\end{document}